\documentclass[12pt]{amsart}
\usepackage{amscd,amsmath,amsthm,amssymb}
\usepackage{amsfonts,amssymb,amscd,amsmath,enumerate,verbatim}
\usepackage[left]{lineno}
\usepackage{pstricks, pst-plot,pst-3d}

\newpsstyle{fatline}{linewidth=1.5pt}
\newpsstyle{fyp}{fillstyle=solid,fillcolor=verylight}
\definecolor{verylight}{gray}{0.97}
\definecolor{light}{gray}{0.9}
\definecolor{medium}{gray}{0.85}
\definecolor{dark}{gray}{0.6}
\def\NZQ{\mathbb}               

\def\ZZ{{\NZQ Z}}
\def\RR{{\NZQ R}}

\def\Dc{{\mathcal D}}
\def\Pc{{\mathcal P}}
\def\Qc{{\mathcal Q}}
\def\Mc{{\mathcal M}}
\def\Hc{{\mathcal H}}

\def\Bc{{\mathcal B}}

\def\ab{{\mathbf a}}

\def\eb{{\mathbf e}}

\def\opn#1#2{\def#1{\operatorname{#2}}} 
\opn\chara{char} \opn\length{\ell} \opn\pd{pd} \opn\rk{rk}
\opn\projdim{proj\,dim} \opn\injdim{inj\,dim} \opn\rank{rank}
\opn\depth{depth} \opn\grade{grade} \opn\height{height}
\opn\embdim{emb\,dim} \opn\codim{codim}
\opn\Cl{Cl}

\opn\Tr{Tr} \opn\bigrank{big\,rank}
\opn\superheight{superheight}\opn\lcm{lcm}
\opn\trdeg{tr\,deg}
	\opn\reg{reg} \opn\lreg{lreg} \opn\ini{in} \opn\lpd{lpd}
	\opn\size{size} \opn\sdepth{sdepth}
	\opn\link{link}\opn\fdepth{fdepth}\opn\lex{lex}
	\opn\tr{tr}
	\opn\type{type}
	\opn\gap{gap}
	\opn\arithdeg{arith-deg}
	\opn\revlex{revlex}
	\opn\div{div} \opn\Div{Div} \opn\cl{cl} \opn\Cl{Cl}
	\opn\Spec{Spec} \opn\Supp{Supp} \opn\supp{supp} \opn\Sing{Sing}
	\opn\Ass{Ass} \opn\Min{Min}\opn\Mon{Mon}
	\opn\Ann{Ann} \opn\Rad{Rad} \opn\Soc{Soc}
	\opn\Im{Im} \opn\Ker{Ker} \opn\Coker{Coker} \opn\Am{Am}
	\opn\Hom{Hom} \opn\Tor{Tor} \opn\Ext{Ext} \opn\End{End}
	\opn\Aut{Aut} \opn\id{id}
	
	\opn\nat{nat}
	\opn\pff{pf}
	\opn\Pf{Pf} \opn\GL{GL} \opn\SL{SL} \opn\mod{mod} \opn\ord{ord}
	\opn\Gin{Gin} \opn\Hilb{Hilb}\opn\sort{sort}
	\opn\PF{PF}\opn\Ap{Ap}
	\opn\mult{mult}
	\opn\bight{bight}
	\opn\div{div}
	\opn\Div{Div}
	\opn\aff{aff}
	\opn\relint{relint} \opn\st{st}
	\opn\lk{lk} \opn\cn{cn} \opn\core{core} \opn\vol{vol}  \opn\inp{inp} \opn\nilpot{nilpot}
	\opn\link{link} \opn\star{star}\opn\lex{lex}\opn\set{set}
	\opn\width{wd}
	\opn\Fr{F}
	\opn\QF{QF}
	\opn\G{G}
	\opn\type{type}\opn\res{res}
	\opn\conv{conv}
	\opn\Deg{Deg}
	\opn\Sym{Sym}
	\opn\Con{Con}
	\opn\gr{gr}
	
	\def\pot#1#2{#1[\kern-0.28ex[#2]\kern-0.28ex]}

	\opn\dirlim{\underrightarrow{\lim}}
	\opn\inivlim{\underleftarrow{\lim}}
	\let\to=\rightarrow
	
	\def\Implies{\ifmmode\Longrightarrow \else
		\unskip${}\Longrightarrow{}$\ignorespaces\fi}
	\def\implies{\ifmmode\Rightarrow \else
		\unskip${}\Rightarrow{}$\ignorespaces\fi}
	\def\iff{\ifmmode\Longleftrightarrow \else
		\unskip${}\Longleftrightarrow{}$\ignorespaces\fi}

	\let\:=\colon
	\newtheorem{Theorem}{Theorem}[section]
	\newtheorem{Lemma}[Theorem]{Lemma}
	\newtheorem{Corollary}[Theorem]{Corollary}

	\newtheorem{Example}[Theorem]{Example}
	
	\newtheorem{Definition}[Theorem]{Definition}
	
	\newtheorem{Conjecture}[Theorem]{Conjecture}
	
	\newtheorem{Question}[Theorem]{Question}
	\let\epsilon\varepsilon
	\let\kappa=\varkappa
	\def\qed{\ifhmode\textqed\fi
		\ifmmode\ifinner\quad\qedsymbol\else\dispqed\fi\fi}
	\def\textqed{\unskip\nobreak\penalty50
		\hskip2em\hbox{}\nobreak\hfil\qedsymbol
		\parfillskip=0pt \finalhyphendemerits=0}
	\def\dispqed{\rlap{\qquad\qedsymbol}}
	
	\opn\dis{dis}
	\def\pnt{{\raise0.5mm\hbox{\large\bf.}}}
	
	\opn\Lex{Lex}

\begin{document}
\title[Pseudo-Gorenstein and level polytopes]{Bounded powers of edge ideals: Pseudo-Gorenstein and Level polytopes}

\author[T.~Hibi]{Takayuki Hibi}
\author[S.~A.~ Seyed Fakhari]{Seyed Amin Seyed Fakhari}

\address{(Takayuki Hibi) Department of Pure and Applied Mathematics, Graduate School of Information Science and Technology, Osaka University, Suita, Osaka 565--0871, Japan}
\email{hibi@math.sci.osaka-u.ac.jp}
\address{(Seyed Amin Seyed Fakhari) Departamento de Matem\'aticas, Universidad de los Andes, Bogot\'a, Colombia}
\email{s.seyedfakhari@uniandes.edu.co}

\subjclass[2020]{52B20, 13H10}

\keywords{Discrete polymatroid, level* polytope, pseudo-Gorenstein* polytope, unimodal sequence}

\begin{abstract}
A lattice polytope $\Pc \subset \RR^n$ of dimension $n$ is called level* if (i) $\Pc$ is normal, (ii) $(\Pc \setminus \partial \Pc) \cap \ZZ^n \neq \emptyset$ and (iii) for each $N = 2,3, \ldots$ and for each $\ab \in N(\Pc \setminus \partial \Pc) \cap \ZZ^n$, there is $\ab_0 \in (\Pc \setminus \partial \Pc) \cap \ZZ^n$ together with $\ab' \in (N-1)\Pc \cap \ZZ^n$ for which $\ab = \ab_0 + \ab'$, where $N\Pc = \{N\ab : \ab \in \Pc\}$.  A normal polytope $\Pc \subset \RR^n$ of dimension $n$ is called {\em pseudo-Gorenstein*} \cite{pseudo} if 
$
|(\Pc \setminus \partial \Pc) \cap \ZZ^n| = 1.
$
A pseudo-Gorenstein* polytope $\Pc$ is level* if and only if $\Pc$ is reflexive up to translation.  In the present paper, level* polytopes together with pseudo-Gorenstein* polytopes arising from discrete polymatroids of bounded powers of edge ideals are studied.      
\end{abstract}	
\maketitle
\thispagestyle{empty}

\section*{Introduction}
A convex polytope $\Pc \subset \RR^n$ of dimension $n$ is called a {\em lattice polytope} if each of whose vertices belongs to $\ZZ^n$.  We say that a lattice polytope $\Pc \subset \RR^n$ of dimension $n$ is {\em normal} if, for each $N = 2,3, \ldots$ and for each $\ab \in N \Pc \cap \ZZ^n$, there are $\ab^{(1)}, \ldots, \ab^{(N)}$ belonging to $\Pc \cap \ZZ^n$ for which $\ab = \ab^{(1)} + \cdots + \ab^{(N)}$, where $N\Pc = \{N\ab : \ab \in \Pc\}$.  

A normal polytope $\Pc \subset \RR^n$ of dimension $n$ with $(\Pc \setminus \partial \Pc) \cap \ZZ^n \neq \emptyset$ is called {\em level*} if, for each $N = 2,3, \ldots$ and for each $\ab \in N(\Pc \setminus \partial \Pc) \cap \ZZ^n$, there exist $\ab_0 \in (\Pc \setminus \partial \Pc) \cap \ZZ^n$ and $\ab' \in (N-1)\Pc \cap \ZZ^n$ for which $\ab = \ab_0 + \ab'$, where $\Pc \setminus \partial \Pc$ is the interior of $\Pc$.  Every reflexive polytope \cite{HH_reflexive} is level*.  

On the other hand, a normal polytope $\Pc \subset \RR^n$ of dimension $n$ is called {\em pseudo-Gorenstein*} \cite{pseudo} if 
$
|(\Pc \setminus \partial \Pc) \cap \ZZ^n| = 1.
$
A pseudo-Gorenstein* polytope $\Pc$ is level* if and only if $\Pc$ is reflexive up to translation.

The terminology {\em level} comes from {\em level rings} introduced by Stanley \cite{Stanley} in the frame of combinatorics.  Every Gorenstein ring is level.  However, a level* polytope is a generalization of a reflexive polytope.  A Gorenstein polytope $\Pc \subset \RR^n$ \cite{HSF_5} is level* if and only if $(\Pc \setminus \partial \Pc) \cap \ZZ^n \neq \emptyset$.  One reason why level* polytopes are of interest is explained in Appendix.  The condition (ii) $(\Pc \setminus \partial \Pc) \cap \ZZ^n \neq \emptyset$ is required to conclude the unimodality (\ref{unimodal}) of $\delta$-vectors.

Let $S=K[x_1, \ldots,x_n]$ denote the polynomial ring  in $n$ variables over a field $K$ with $n \geq 3$.  If $u \in S$ is a monomial, then $M_{\leq u}$ stands for the set of those monomials $w \in S$ which divide $u$.  In particular, $1 \in M_{\leq u}$ and $u \in M_{\leq u}$.  Let $G$ be a finite graph on the vertex set $V(G)=\{x_1, \ldots, x_n\}$, where $n \geq 2$, with no loop, no multiple edge and no isolated vertex, and $E(G)$ the set of edges of $G$.  The edge ideal of $G$ is the ideal $I(G) \subset S$ which is generated by those $x_ix_j$ with $\{x_i, x_j\} \in E(G)$. 

Let $\ZZ_{>0}$ denote the set of positive integers. Given $\mathfrak{c}=(c_1,\ldots,c_n)\in (\ZZ_{>0})^n$ and an integer $q\geq 1$, we denote by $(I(G)^q)_\mathfrak{c}$ the ideal of $S$ generated by those monomials $x_1^{a_1}\cdots x_n^{a_n} \in I(G)^q$ with each $a_i \leq c_i$.  Let $\delta_{\mathfrak{c}}(I(G))$ denote the biggest integer $q$ for which $(I(G)^q)_\mathfrak{c} \neq (0)$.  Then $(I(G)^{\delta_{\mathfrak{c}}(I(G))})_\mathfrak{c}$ is a polymatroidal ideal \cite[Theorem 4.3]{HSF_1}.  

Let $\Bc(G,\mathfrak{c})$ denote the minimal set of monomial generators of $(I(G)^{\delta_{\mathfrak{c}}(I(G))})_\mathfrak{c}$.  Set $\Mc(G,\mathfrak{c}):= \{ M_{\leq u} : u \in \Bc(G,\mathfrak{c})\}$ and
\[
\Dc(G,\mathfrak{c}):=\{(a_1, \ldots, a_n) \in \ZZ^n : x_1^{a_1}\cdots x_n^{a_n} \in \Mc(G,\mathfrak{c})\}.
\]
The standard unit coordinate vectors $\eb_1, \ldots, \eb_n$ of $\RR^n$ together with the origin $(0,\ldots,0)$ of $\RR^n$ belong to $\Dc(G,\mathfrak{c})$.  Now, since $(I(G)^{\delta_{\mathfrak{c}}(I(G))})_\mathfrak{c}$ is a polymatroidal ideal, it follows from \cite[Theorem 2.3]{HH_discrete} that $\Dc(G,\mathfrak{c})$ is a discrete polymatroid \cite[Definition 2.1]{HH_discrete}.  Let $\conv(\Dc(G,\mathfrak{c})) $ denote the convex hull of $\Dc(G,\mathfrak{c})$ in $\RR^n$.  It then follows from \cite[Theorem 3.4]{HH_discrete} that $\conv(\Dc(G,\mathfrak{c}))$ is a polymatroid \cite[Definition 1.1]{HH_discrete}.  In particular, $\conv(\Dc(G,\mathfrak{c}))$ is normal \cite[Theorem 6.1]{HH_discrete}. 

Let $2^{[n]}$ denote the set of subsets of $[n]:=\{1, \ldots, n\}$.  The {\em ground set rank function} \cite[p.~243]{HH_discrete} $\rho_{(G,\mathfrak{c})} : 2^{[n]} \to \ZZ_{>0}$ of $\conv(\Dc(G,\mathfrak{c}))$ is defined by setting
\[
\rho_{(G,\mathfrak{c})}(X) = \max\left\{\sum_{i \in X}a_i : x_1^{a_1}\cdots x_n^{a_n} \in \Bc(G,\mathfrak{c})\right\}
\]
for $\emptyset \neq X \subset [n]$ together with $\rho_{(G,\mathfrak{c})}(\emptyset)=0$. A nonempty subset $A\subset [n]$ is called {\em $\rho_{(G,\mathfrak{c})}$-closed} if for any $B\subset [n]$ with $A \subsetneq B$, one has $\rho_{(G,\mathfrak{c})}(A) < \rho_{(G,\mathfrak{c})}(B)$. A nonempty subset $A\subset [n]$ is called {\em $\rho_{(G,\mathfrak{c})}$-separable} if there exist nonempty subsets $A'$ and $A''$ of $[n]$ with $A = A' \cup A''$ and $A' \cap A'' = \emptyset$ for which $\rho_{(G,\mathfrak{c})}(A) = \rho_{(G,\mathfrak{c})}(A') + \rho_{(G,\mathfrak{c})}(A'')$.  For $\emptyset \neq A \subset [n]$, we define the hyperplane $\Hc_A \subset \RR^n$ by
\[
\Hc_A = \{ (x_1,\ldots,x_n) \in \RR^n : \sum_{i\in A} x_i = \rho(A) \}.
\]  
For $i\in [n]$, we define the hyperplane $\Hc^{(i)}  \subset \RR^n$ by
\[
\Hc^{(i)} = \{ (x_1,\ldots,x_n) \in \RR^n : x_i = 0 \}.
\] 
Now, Edmonds \cite{Edmonds} says that
 
\begin{Lemma} \label{facets}
The facets of $\conv(\Dc(G,\mathfrak{c})) \subset \RR^n$ are all $\Hc^{(i)} \cap \conv(\Dc(G,\mathfrak{c}))$ with $i \in [n]$ and all $\Hc_A \cap \conv(\Dc(G,\mathfrak{c}))$, where $A$ ranges over all $\rho_{(G,\mathfrak{c})}$-closed and $\rho_{(G,\mathfrak{c})}$-inseparable subsets of $[n]$.
\end{Lemma}

In the present paper, starting an exhibition of level* polytopes and nonlevel* polytopes in Section $1$, in Section $2$, we study pseudo-Gorenstein* and level* polytopes arising from complete bipartite graphs.  A finite graph $G$ on $n$ vertices is called  {\em labeling pseudo-Gorenstein*} if there is $\mathfrak{c} \in (\ZZ_{>0})^{n}$ for which $\conv(\Dc(G,\mathfrak{c})) \subset \RR^n$ is pseudo-Gorenstein*.   Labeling pseudo-Gorenstein* complete bipartite graphs are $K_{n,m}$ with $n\leq m\leq 2n-1$ (Theorem \ref{complete_bipartite}).  Theorem \ref{compbip-level} is a numerical characterization for $\conv(\Dc(K_{m,n},\mathfrak{c})$ to be level*.  In Section $3$, we discuss convex polytopes of Veronese type (Definition \ref{veronese_type}).  Theorem \ref{veronese_criterion} is a numerical characterization for a convex polytope of Veronese type to be level*.  Finally, in Section $4$, a complete classification of labeling pseudo-Gorenstein* trees is obtained (Theorem \ref{treeclassify}).

\section{Examples of level* and nonlevel* polytopes}

We start an exhibition of level* polytopes and nonlevel* polytopes.  We refer the reader to \cite[Section 1]{HSF_1} and \cite[Section 1]{HSF_5} for basic terminologies on finite graphs. 

\begin{Example}
{\em
Let $G$ be the finite graph on $\{x_1,x_2,x_3\}$, $E(G) =\{ \{x_1,x_2\},\{x_2,x_3\} \}$ and $\mathfrak{c} = (2,3,2)$.  One easily sees that $\conv(\Dc(G,\mathfrak{c})) \subset \RR^3$ is defined by the system of linear inequalities
$
0 \leq x_1 \leq 2, 0 \leq x_2 \leq 3, 0 \leq x_3 \leq 2$ and $x_1+x_3 \leq 3.
$  
By virtue of \cite[Theorem 4.2]{BH}, it follows that $\conv(\Dc(G,\mathfrak{c}))$ is level*.
} 
\end{Example}

\begin{Example}
{\em 
If $K_{3,4}$ is the complete bipartite graph on $\{x_1,x_2,x_3\}\sqcup \{x_4,x_5,x_6,x_7\}$ and $\mathfrak{c} = (2,\ldots,2) \in (\ZZ_{>0})^7$, then $\conv(\Dc(G,\mathfrak{c}))$ is defined by the system of linear inequalities $0\leq x_i \leq 2$ for $1 \leq i \leq 7$ together with $x_4+x_5+x_6+x_7\leq 6$. Thus.
    \[
    (\conv(\Dc(k_{3,4},\mathfrak{c})) \setminus \partial \conv(\Dc(K_{3,4},\mathfrak{c}))) \cap \ZZ^7 = \{(1,\ldots,1)\}.
    \]
    Since $|(\conv(\Dc(G,\mathfrak{c})) \setminus \partial \conv(\Dc(G,\mathfrak{c}))) \cap \ZZ^7|=1$ and since $\conv(\Dc(G,\mathfrak{c})$ is not reflexive, it follows that $\conv(\Dc(G,\mathfrak{c})$ is not level*.  In fact, $(3,3,3,3,3,3,2)$ belongs to $2(\conv(\Dc(G,\mathfrak{c})) \setminus \partial \conv(\Dc(G,\mathfrak{c}))) \cap \ZZ^7$, however, $$(3,3,3,3,3,3,2)-(1,1,1,1,1,1,1)=(2,2,2,2,2,2,1)$$ does not belong to $\conv(\Dc(G,\mathfrak{c}))$.
    }
\end{Example}

Lemma \ref{dilation} might be indispensable in our forthcoming project on lattice polytopes.

\begin{Lemma}
\label{dilation}
Let $G$ be a finite graph on $[n]$ and $\mathfrak{c}\in (\ZZ_{>0})^n$.  One has
        \[   N\conv(\Dc(G,\mathfrak{c}))\subset \conv(\Dc(G,N\mathfrak{c})), \, \, \, \, \, N=1,2,\ldots.
    \]
\end{Lemma} 

\begin{proof}
Let $\Bc(G,\mathfrak{c}) = \{u_1, \ldots, u_s\}$.
It follows that $(a_1, \ldots, a_n) \in \ZZ^n$ belongs to $N \Dc(G,\mathfrak{c})$ if and only if $x_1^{a_1}\cdots x_n^{a_n}$ can divide one of the monomials $u_1^N, \ldots, u_s^N$.  Since $\{u_1^N, \ldots, u_s^N\} \subset \Dc(G,N\mathfrak{c}))$, we have $N\conv(\Dc(G,\mathfrak{c})) = \conv(N \Dc(G,\mathfrak{c})) \subset \conv(\Dc(G,N\mathfrak{c}))$, as desired.
\, \, \, \, \, \, \, \,
\end{proof}

\begin{Example}
{\em 
Let $G$ be the triangle on $\{x_1,x_2,x_3\}$ and 
$\mathfrak{c} = (1,1,1)$.  Then $\conv(\Dc(G,\mathfrak{c}))$ is defined by the system of linear inequalities $0\leq x_1 \leq 1, 0\leq x_2 \leq 1, 0\leq x_3 \leq 1$ and $x_1+x_2+x_3 \leq 2$.  Hence $2\conv(\Dc(G,\mathfrak{c}))$ is defined by the system of linear inequalities $0\leq x_1 \leq 2, 0\leq x_2 \leq 2, 0\leq x_3 \leq 2$ and $x_1+x_2+x_3 \leq 4$.  On the other hand, $\conv(\Dc(G,2\mathfrak{c}))$ is defined by defined by the system of linear inequalities $x_1 \leq 2, x_2 \leq 2, x_3 \leq 2$.  It then follows that $2\conv(\Dc(G,\mathfrak{c}))\subsetneq\conv(\Dc(G,2\mathfrak{c}))$.
}
\end{Example}

In the following sections, to simplify the notation, We sometimes use $[n]=\{1, \ldots, n\}$, instead of $V(G) = \{x_1, \ldots, x_n\}$, for the vertex set of a finite graph $G$. 

\section{Complete bipartite graphs}
Let $G$ be a finite graph on $[n]$.  We say that $G$ is {\em labeling pseudo-Gorenstein*} if there is $\mathfrak{c} \in (\ZZ_{>0})^{n}$ for which $\conv(\Dc(G,\mathfrak{c})) \subset \RR^n$ is pseudo-Gorenstein*.  Every Hamiltonian graph is labeling pseudo-Gorenstein* (\cite[Example 4.3]{HSF_5}).  In particular, every cycle, every complete graph and every complete bipartite graph $K_{n,n}$ with $n \geq 2$ is labeling pseudo-Gorenstein*.  Every finite graph with a perfect matching is labeling pseudo-Gorenstein* (\cite[Example 4.3]{HSF_5}). Every regular graph is labeling pseudo-Gorenstein* (\cite[Theorem 7.4]{HSF_5}).

\begin{Lemma} \label{compbip-poly1}
Let $G=K_{m,n}$ be a complete bipartite graph on $[m+n]$ with vertex partition $[m]\sqcup ([m+n]\setminus [m])$ and $\mathfrak{c}=(c_1, \ldots, c_{m+n})\in (\ZZ_{>0})^{m+n}$. If $c_{1}+\cdots +c_m=c_{m+1}+\cdots +c_{m+n}$, then $\conv(\Dc(G,\mathfrak{c})) \subset \RR^{m+n}$ is the cube defined by the system of inequalities $0\leq x_i\leq c_i, 1 \leq i \leq  m+n$.
\end{Lemma}

\begin{proof}
Sunce $\Bc(G,\mathfrak{c}) = \{x_1^{c_1}x_2^{c_2}\cdots x_{m+n}^{c_{m+n}}\}$, the assertion follows.
\, \, \, \, \, \, \, \, \, \,
\end{proof}

\begin{Lemma} \label{compbip-poly2}
Let $G=K_{m,n}$ be a complete bipartite graph on $[m+n]$ with vertex partition $[m]\sqcup ([m+n]\setminus [m])$ and $\mathfrak{c}=(c_1, \ldots, c_{m+n})\in (\ZZ_{>0})^{m+n}$.  Suppose that $c_1+\cdots +c_m>c_{m+1}+\cdots +c_{m+n}$ and $c_i\leq c_{m+1}+\cdots +c_{m+n}$ for each $i\in [m]$.  Set$$A:=\{i\in [m] : c_i=c_{m+1}+\cdots +c_{m+n}\}$$and $B:=[m]\setminus A$. Then $\conv(\Dc(G,\mathfrak{c})) \subset \RR^{m+n}$ is defined by the system of linear inequalities 

$\bullet$ $0\leq x_i\leq c_i, \, \, \, \, \, i=m+1, \cdots, m+n$;

$\bullet$ $0\leq x_i\leq c_i, \, \, \, \, \, i\in B$;

$\bullet$ $0\leq x_i, \, \, \, \, \,  i\in A$;

$\bullet$ $x_1+\cdots +x_m\leq c_{m+1}+\cdots +c_{m+n}$.
\end{Lemma}

\begin{proof}
Since $\Bc(G,\mathfrak{c})$ consists of those monomials of the form $x_{m+1}^{c_{m+1}}\cdots x_{m+n}^{c_{m+n}}u$, where $u$ is a $(c_1, \ldots, c_m)$-bounded monomial of degree $c_{m+1}+\cdots +c_{m+n}$ on variables $x_1, \ldots, x_m$, it follows that the singletons $\{m+1\}, \ldots, \{m+n\}$ are $\rho_{(G,\mathfrak{c})}$-closed and $\rho_{(G,\mathfrak{c})}$-inseparable.  In addition, the singleton $\{i\}$ is $\rho_{(G,\mathfrak{c})}$-closed and $\rho_{(G,\mathfrak{c})}$-inseparable for each $i\in B$ and the singleton $\{j\}$ is not $\rho_{(G,\mathfrak{c})}$-closed for each $j\in A$.  

Clearly $[m]$ is $\rho_{(G,\mathfrak{c})}$-closed. We show that $[m]$ is $\rho_{(G,\mathfrak{c})}$-inseparable. One has $\rho_{(G,\mathfrak{c})}([m])=c_{m+1}+\cdots +c_{m+n}$.   Let $A_1$ and $A_2$ be nonempty subsets of $[m]$ with $A_1 \cap A_2 = \emptyset$ and $A_1\cup A_2=[m]$.  If $\sum_{i\in A_1}c_i\geq c_{m+1}+\cdots +c_{m+n}$, then $\rho_{(G,\mathfrak{c})}(A_1)=c_{m+1}+\cdots +c_{m+n}$.  Consequently,
\[
\rho_{(G,\mathfrak{c})}(A_1)+\rho_{(G,\mathfrak{c})}(A_2) > c_{m+1}+\cdots +c_{m+n}=\rho_{(G,\mathfrak{c})}([m]).
\]
So, let $\sum_{i\in A_1}c_i< c_{m+1}+\cdots +c_{m+n}$.  Similarly, $\sum_{i\in A_2}c_i < c_{m+1}+\cdots +c_{m+n}$.  Then $\rho_{(G,\mathfrak{c})}(A_j)=\sum_{i\in A_j}c_i$ for $j=1,2$.  Furthermore,  
$$\rho_{(G,\mathfrak{c})}(A_1)+\rho_{(G,\mathfrak{c})}(A_2)=\sum_{i=1}^mc_i>c_{m+1}+\cdots +c_{m+n}=\rho_{(G,\mathfrak{c})}([m]).$$   
Hence, $[m]$ is $\rho_{(G,\mathfrak{c})}$-inseparable, as desired. 

Now, suppose that $X$ is a $\rho_{(G,\mathfrak{c})}$-closed and $\rho_{(G,\mathfrak{c})}$-inseparable subset of $[m+n]$ with $|X|\geq 2$. If $k\in X$ for some $k\in \{m+1, \ldots, m+n\}$, then$$\rho_{(G,\mathfrak{c})}(\{k\})+\rho_{(G,\mathfrak{c})}(X\setminus \{k\})=\rho_{(G,\mathfrak{c})}(X),$$which is a contradiction. Thus, $k\notin X$ for each $k\in \{m+1, \ldots, m+n\}$. In other words, $X\subset [m]$. We show that $X=[m]$.  By contradiction, suppose that $X$ is a proper subset of $[m]$. If $\rho_{(G,\mathfrak{c})}(X) <c_{m+1}+\cdots +c_{m+n}$, then$$\rho_{(G,\mathfrak{c})}(X)=\sum_{i\in X}c_i.$$It then follows that, for any pair of nonempty sets $X_1$ and $X_2$ with $X_1 \cap X_2 = \emptyset$ and $X=X_1\cup X_2$, one has $$\rho_{(G,\mathfrak{c})}(X_1)+\rho_{(G,\mathfrak{c})}(X_2)=\sum_{i\in X_1}c_i+\sum_{i\in X_1}c_i=\sum_{i\in X}c_i=\rho_{(G,\mathfrak{c})}(X),$$a contradiction. So, $\rho_{(G,\mathfrak{c})}(X)=c_{m+1}+\cdots +c_{m+n}$.  However, this is impossible, as $[m]$ properly contains $X$ and $\rho_{(G,\mathfrak{c})}([m])=\rho_{(G,\mathfrak{c})}(X)=c_{m+1}+\cdots +c_{m+n}$.

So far, we proved that $\rho_{(G,\mathfrak{c})}$-closed and $\rho_{(G,\mathfrak{c})}$-inseparable subsets of $[m+n]$ are exactly the singletons $\{m+1\}, \ldots \{m+n\}$, the singletons $\{i\}$ for each $i\in B$ together with the set $[m]$.  Finally, Lemma \ref{facets} completes the proof.
\hspace{3.2cm}
\end{proof}

\begin{Corollary}
    \label{compbip_inner}
Let $G=K_{m,n}$ be a complete bipartite graph on $[m+n]$ with vertex partition $[m]\sqcup ([n+m]\setminus [m])$ and $\mathfrak{c}=(c_1, \ldots, c_{m+n})\in (\ZZ_{>0})^{m+n}$. Suppose that $c_1+\cdots +c_m>c_{m+1}+\cdots +c_{m+n}$ and $c_i\leq c_{m+1}+\cdots +c_{m+n}$ for each $i\in [m]$. Set$$A:=\{i\in [m] \mid c_i=c_{m+1}+\cdots +c_{m+n}\}$$and $B:=[m]\setminus A$. Then $(\conv(\Dc(G,\mathfrak{c})) \setminus \partial \conv(\Dc(G,\mathfrak{c}))) \cap \ZZ^{n+1} \neq \emptyset$ if and only if $c_i \geq 2$ for each $i\in ([m+n])\setminus [m])\cup B$ and $c_{m+1}+\cdots +c_{m+n} \geq m+1$. 
\end{Corollary}

\begin{Theorem}
    \label{complete_bipartite}
The labeling pseudo-Gorenstein* complete bipartite graphs are $K_{n,m}$ with $n\in \ZZ_{>0}$ and $n\leq m\leq 2n-1$.
\end{Theorem}

\begin{proof}
Let $G=K_{m,n}$ be a complete bipartite graph and $[m]\sqcup ([m+n]\setminus [m])$ its vertex partition. If $n\leq m\leq 2n-1$ and if $\mathfrak{c}=(2, \ldots, 2)\in (\ZZ_{>0})^{m+n}$, then it follows from Lemmas \ref{compbip-poly1} and \ref{compbip-poly2} that  $\conv(\Dc(G,\mathfrak{c})) \subset \RR^n$ is pseudo-Gorenstein*.

Suppose that $G$ is labeling pseudo-Gorenstein* and $\mathfrak{c}=(c_1, \ldots, c_{m+n})\in (\ZZ_{>0})^{m+n}$ satisfies the condition that  $\conv(\Dc(G,\mathfrak{c})) \subset \RR^n$ is pseudo-Gorenstein*. 

If $c_1+\cdots +c_m=c_{m+1}+\cdots +c_{m+n}$, then by using Lemma \ref{compbip-poly1} one has $c_1=\cdots =c_{m+n}=2$.  Hence $c_1+\cdots +c_m=c_{m+1}+\cdots +c_{m+n}$ implies that $m=n$.

Let $c_1+\cdots +c_m\neq c_{m+1}+\cdots +c_{m+n}$.  One may assume that $c_1+\cdots +c_m> c_{m+1}+\cdots +c_{m+n}$. Furthermore, one may also assume that $c_i\leq c_{m+1}+\cdots +c_{m+n}$ for each $i\in [m]$. Hence, $\conv(\Dc(G,\mathfrak{c}))$ coincides with the convex polytope described in Lemma \ref{compbip-poly2}.  One has  $(1,1,...,1) \in (\conv(\Dc(G,\mathfrak{c}) \setminus \partial \conv(\Dc(G,\mathfrak{c})) \cap \ZZ^{m+n}$ and therefore, $c_i\geq 2$, for each $i=m+1, \ldots, m+n$.  If there is $m+1 \leq i \leq m+n$ with $c_i\geq 3$, then $(1,1,...,1) + \eb_i \in \ZZ^{m+n}$ belongs to the interior of $\conv(\Dc(G,\mathfrak{c}))$, a contradiction. So, $c_i=2$, for each $i=m+1, \ldots, m+n$.  Since $(1,1,...,1)$ belongs to the interior of $\conv(\Dc(G,\mathfrak{c}))$, by using the last inequality of Lemma \ref{compbip-poly2}, one has $m\leq 2n-1$.  We show that $m\geq n$.  Let $m<n$.  It then follows from $$c_1+\cdots +c_m> c_{m+1}+\cdots +c_{m+n}=2n$$ that there is $j\in [m]$ with $c_j\geq 3$. Then $(1,1,...,1)+\eb_j\in \ZZ^{m+n}$ belongs to the interior of $\conv(\Dc(G,\mathfrak{c}))$, a contradiction.  Thus, $m\geq n$, as desired.
\hspace{9.2cm}
\end{proof}

We now come to a numerical characterization for $\conv(\Dc(K_{m,n},\mathfrak{c})$ to be level*.

\begin{Theorem}
\label{compbip-level} 
Let $G=K_{m,n}$ be a complete bipartite graph on $[m+n]$ with vertex partition $[m]\sqcup ([n+m]\setminus [m])$ and $\mathfrak{c}=(c_1, \ldots, c_{m+n})\in (\ZZ_{>0})^{m+n}$. Suppose that $c_1+\cdots +c_m>c_{m+1}+\cdots +c_{m+n}\geq m+1$ and $c_i\leq c_{m+1}+\cdots +c_{m+n}$ for each $i\in [m]$. Set$$A:=\{i\in [m] : c_i=c_{m+1}+\cdots +c_{m+n}\}$$and $B:=[m]\setminus A$, and suppose that $c_i\geq 2$ for each $i\in B\cup ([m+n]\setminus [m])$.  Then $\conv(\Dc(G,\mathfrak{c})$, which is defined by the system of linear inequalities described in Lemma \ref{compbip-poly2}, is level* if and only if the following conditions are satisfied.
 \begin{itemize}
\item[(i)]
There is no subset $\emptyset \neq X\subset B$ for which 
\begin{eqnarray*}
\label{ABCDEFG}
c_{m+1}+\cdots + c_{m+n}<\sum_{i\in X}c_i< c_{m+1}+\cdots + c_{m+n}-m+2|X|-1.  
\end{eqnarray*}

\item[(ii)]
There is no subset $\emptyset \neq X\subset [m]$ for which 
\begin{eqnarray*}
\label{abcdefg}
\sum_{i=1}^{m}c_i-\sum_{i\in X}c_i < c_{m+1}+\cdots + c_{m+n}\leq 2|X|+\sum_{i=1}^{m}c_i-\sum_{i\in X}c_i-m. 
\end{eqnarray*}
\end{itemize}
\end{Theorem}

\begin{proof}
{\bf (``only if'')}  First, suppose that (i) fails to be satisfied.  In other words, there exists $\emptyset \neq X\subset B$ for which $$c_{m+1}+\cdots + c_{m+n}<\sum_{i\in X}c_i< c_{m+1}+\cdots + c_{m+n}-m+2|X|-1.$$
It follows that
\begin{align*}
& c_{m+1}+\cdots + c_{m+n}+\sum_{i\in X}c_i-|X|\\  < & \, \min\big\{2(c_{m+1}+\cdots + c_{m+n})-1-m+|X|, \sum_{i\in X}(2c_i-1)\big\}.  
\end{align*}
Hence, for each $i\in X$, one can find an integer $0 < a_i\leq 2c_i-1$ with
\begin{align*}
& \, c_{m+1}+\cdots + c_{m+n}+\sum_{i\in X}c_i-|X|+1 \\
\leq & \, \sum_{i\in X}a_i
\leq  \, 2(c_{m+1}+\cdots + c_{m+n})-1-m+|X|.   
\end{align*}
Set $\mathfrak{b}:=(b_1, \ldots, b_{m+n})$, where $b_i=a_i$ if $i\in X$ and $b_i=1$ if $i\in [m+n]\setminus X$.  Let $\Pc = \conv(\Dc(G,\mathfrak{c}))$.  It follows from the choice of $a_i$'s that $\mathfrak{b}\in 2(\Pc \setminus \partial \Pc)\cap \ZZ^{m+n}$. Suppose that $\Pc$ is level*.  Then there exist $\mathfrak{b'}=(b_1', \ldots, b_{m+n}')\in \Pc\cap \ZZ^{m+n}$ and $\mathfrak{b''}=(b_1'', \ldots, b_{m+n}'')\in (\Pc \setminus \partial \Pc) \cap \ZZ^{m+n}$ with $\mathfrak{b}=\mathfrak{b'}+\mathfrak{b''}$. Note that $b_i''=1$ if $i\in [m+n]\setminus X$ and $b_i''\leq c_i-1$ if $i\in X$.  Hence, $$b_1''+\cdots +b_m''\leq (m-|X|)+\sum_{i\in X}(c_i-1)=m-2|X|+\sum_{i\in X}c_i.$$  Since $\mathfrak{b'}\in \Pc$, one has $b_1'+\cdots +b_m'\leq c_{m+1}+\cdots + c_{m+n}$. Thus,
\begin{align*}
&\,b_1+\cdots +b_m\\  =&\, (b_1'+b_1'')+\cdots +(b_m'+b_m'')\\  =& \, (b_1'+\cdots +b_m')+(b_1''+\cdots +b_m'')\\  \leq&\,  c_{m+1}+\cdots + c_{m+n}+m-2|X|+\sum_{i\in X}c_i
\end{align*}
which is a contradiction, as
\begin{align*}
& \, b_1+\cdots +b_m
\\
= & \, m-|X|+\sum_{i\in X}a_i\\  \geq & \,c_{m+1}+\cdots + c_{m+n}+m-2|X|+\sum_{i\in X}c_i+1.
\end{align*}

Second, suppose that (ii) fails to be satisfied. Let $\emptyset \neq X\subset [m]$ satisfy $$\sum_{i=1}^{m}c_i-\sum_{i\in X}c_i < c_{m+1}+\cdots + c_{m+n}\leq 2|X|+\sum_{i=1}^{m}c_i-\sum_{i\in X}c_i-m.$$ It follows from the second inequality that $$c_1+\cdots +c_m\geq c_{m+1}+\cdots + c_{m+n}+m.$$ Set $X':=[m]\setminus X$.  So, $\sum_{i=1}^{m}c_i-\sum_{i\in X}c_i<c_{m+1}+\cdots + c_{m+n}$ is equivalent to $$\sum_{i\in X'}c_i<c_{m+1}+\cdots + c_{m+n}.$$  One can find an integer $N\geq 2$ with $$|X|+\sum_{i\in X'}[Nc_i-1]\leq N(c_{m+1}+\cdots + c_{m+n})-1.$$Recall that $\Pc = \conv(\Dc(G,\mathfrak{c})$.  There is $\mathfrak{a}=(a_1, \ldots, a_{m+n})\in N(\Pc\setminus\partial\Pc)\cap \ZZ^{m+n}$ for which $a_i=Nc_i-1$ for each $i\in X'$ and $$a_1+\cdots +a_{m}=N(c_{m+1}+\cdots + c_{m+n})-1.$$  Suppose that $\Pc$ is level*. Then there are $\mathfrak{a'}=(a_1', \ldots, a_{m+n}')\in (N-1)\Pc\cap \ZZ^{m+n}$ and $\mathfrak{a''}=(a_1'', \ldots, a_{m+n}'')\in (\Pc\setminus\partial\Pc)\cap \ZZ^{m+n}$ with $\mathfrak{a}=\mathfrak{a'}+\mathfrak{a''}$. Then
\begin{align*}
a_1''+\cdots +a_m'' < & \, c_{m+1}+\cdots + c_{m+n} \\ \leq & \, 2|X|+\sum_{i=1}^{m}c_i-\sum_{i\in X}c_i-m\\ =& \, 2|X|+\sum_{i\in X'}c_i-m =|X|+\sum_{i\in X'}(c_i-1).  
\end{align*}
Since $a_i''\geq 1$ for each $i\in X$, it follows that$$\sum_{i\in X'}a_i''<\sum_{i\in X'}(c_i-1).$$
As $a_i=Nc_i-1$ for each $i\in X'$, we deduce that
\begin{align*}
\sum_{i\in X'}a_i'=&\,\sum_{i\in X'}a_i-\sum_{i\in X'}a_i''\\>&\,\sum_{i\in X'}(Nc_i-1)-\sum_{i\in X'}(c_i-1) \\=&\, (N-1)\sum_{i\in X'}c_i 
\end{align*}
which is a contradiction, as $\mathfrak{a'}\in (N-1)\Pc$.

\medskip

\noindent
{\bf (``if'')}  
Note that $\Pc=\conv(\Dc(G,\mathfrak{c}))$ is normal with $(1, \ldots, 1)\in (\Pc\setminus\partial\Pc)\cap \ZZ^{m+n}$.  Let $\mathfrak{a}=(a_1, \ldots, a_{m+n})\in N(\Pc\setminus \partial\Pc) \cap \ZZ^{m+n}$, where $2 \leq N \in \ZZ$.  Set $$b_i:=\max\{0, a_i-c_i+1\},$$for each $i\in [m+n]$ and $\mathfrak{b}:=(b_1, \ldots, b_{m+n})$. We prove $\mathfrak{b}\in (N-1)\Pc$. One has  

$\bullet$ $0\leq b_i\leq (N-1)c_i, \, \, \, \, \,i=m+1, \ldots, m+n$;

$\bullet$ $0\leq b_i\leq (N-1)c_i, \, \, \, \, \,i\in B$;

$\bullet$ $0\leq b_i, \, \, \, \, \, i\in A$.

\noindent
In order to prove $\mathfrak{b}\in (N-1)\Pc$, we need to show that$$b_1+\cdots +b_m\leq (N-1)(c_{m+1}+\cdots +c_{m+n}).$$If there is $i \in A$ with $b_i\geq 1$, then $$a_i=b_i+c_i-1=b_i+(c_{m+1}+\cdots +c_{m+n})-1.$$Thus,
\begin{align*}
 b_1+\cdots +b_m  \leq & \, a_1+\cdots + a_m-(c_{m+1}+\cdots +c_{m+n})+1\\  \leq & \, (N-1)(c_{m+1}+\cdots +c_{m+n}).   
\end{align*}
So, suppose that $b_i=0$ for each $i\in A$ and that  $b_{j_1}, \ldots, b_{j_s}$ are nonzero elements of $\{b_i:i\in B\}$. Then $a_{j_k}=b_{j_k}+(c_{j_k}-1)$ for each $k=1, \ldots, s$ and$$b_1+\cdots +b_m=b_{j_1}+\cdots+b_{j_s}.$$If$$a_{j_1}+\cdots +a_{j_s}\leq (c_{j_1}+\cdots +c_{j_s})+(N-1)(c_{m+1}+\cdots +c_{m+n})-s,$$ then
\begin{align*}
b_{j_1}+\cdots+b_{j_s}  =&\,a_{j_1}+\cdots+a_{j_s}-(c_{j_1}+\cdots + c_{j_s})+s\\ \leq & \, (N-1)(c_{m+1}+\cdots +c_{m+n}).  
\end{align*}
Therefore, suppose that$$a_{j_1}+\cdots +a_{j_s}> (c_{j_1}+\cdots +c_{j_s})+(N-1)(c_{m+1}+\cdots +c_{m+n})-s.$$

\smallskip

{\bf (Claim 1.)} $c_{j_1}+\cdots +c_{j_s}< (c_{m+1}+\cdots +c_{m+n})-m+2s-1$.

\smallskip

\noindent
In fact, since $$a_1+\cdots +a_m\leq N(c_{m+1}+\cdots +c_{m+n})-1$$ and since $a_i\geq 1$ for each $i\in [m]$, one has 
\begin{align*}
& \, c_{j_1}+\cdots +c_{j_s}<(a_{j_1}+\cdots+a_{j_s})-(N-1)(c_{m+1}+\cdots +c_{m+n})+s\\  \leq & \, N(c_{m+1}+\cdots +c_{m+n})-1-(m-s)-(N-1)(c_{m+1}+\cdots +c_{m+n})+s\\ =& \, (c_{m+1}+\cdots +c_{m+n})-m+2s-1,  
\end{align*}
as desired.

\smallskip

{\bf (Claim 2.)} $c_{j_1}+\cdots +c_{j_s}> c_{m+1}+\cdots +c_{m+n}$.

\smallskip

\noindent
In fact, since $a_i\leq Nc_i-1$ for each $i\in B$, one has
\begin{align*}
c_{j_1}+\cdots +c_{j_s} &<(a_{j_1}+\cdots+a_{j_s})-(N-1)(c_{m+1}+\cdots +c_{m+n})+s\\ & \leq N(c_{j_1}+\cdots +c_{j_s})-s-(N-1)(c_{m+1}+\cdots +c_{m+n})+s\\ & =N(c_{j_1}+\cdots +c_{j_s})-(N-1)(c_{m+1}+\cdots +c_{m+n}).   
\end{align*}
Consequently$$(N-1)(c_{m+1}+\cdots +c_{m+n})<(N-1)(c_{j_1}+\cdots +c_{j_s}),$$
as desired.

\smallskip

Setting $X=\{j_1, \ldots, j_s\}$ and combining Claims 1 and 2, one has $$c_{m+1}+\cdots +c_{m+n}<\sum_{i\in X}c_i< (c_{m+1}+\cdots +c_{m+n})-m+2|X|-1,$$which is a contradiction. Hence $\mathfrak{b}\in (N-1)\Pc$.

Now, for each $i\in [m+n]$, set $b_i':=a_i-b_i$. Let $\mathfrak{b'}=(b_1', \ldots, b_{m+n}')$. Thus, $\mathfrak{a}=\mathfrak{b}+\mathfrak{b'}$. So, if $\mathfrak{b'}\in \Pc\setminus\partial\Pc$, then we are done.  Suppose that $\mathfrak{b'}\notin \Pc\setminus\partial\Pc$.  One has

$\bullet$ $0< b_i'< c_i, \, \, \, \, \, i=m+1\ldots, m+n$;

$\bullet$ $0< b_i'< c_i, \, \, \, \, \, i\in B$;

$\bullet$ $0< b_i', \, \, \, \, \, i\in A$.

\noindent
It follows from $\mathfrak{b'}\notin \Pc\setminus\partial\Pc$ that $$b_1'+\cdots +b_m'\geq c_{m+1}+\cdots +c_{m+n}.$$ Hence, there is an integer $\ell \geq 0$ with $$b_1'+\cdots +b_m'=c_{m+1}+\cdots +c_{m+n}+\ell.$$ It follows from$$a_1+\cdots +a_m\leq N(c_{m+1}+\cdots +c_{m+n})-1$$that$$b_1+\cdots +b_m\leq (N-1)(c_{m+1}+\cdots +c_{m+n})-\ell-1.$$Set$$D:=\{i\mid i\in [m], a_i\leq (N-1)c_i+1\}$$and set $D':=[m]\setminus D$. The following two cases arise.

\smallskip

{\bf (Case 1.)} Suppose that$$\sum_{i\in D}(a_i-b_i-1)+\sum_{i\in D'}[(N-1)c_i-b_i]\geq\ell+1.$$For each $i\in [m]$, one can choose an integer $d_i\geq b_i$ for which $d_i\leq a_i-1$ if $i\in D$ and $d_i\leq (N-1)c_i$ if $i\in D'$, such that $$\sum_{i=1}^{m}d_i=\sum_{i=1}^{m}b_i+\ell+1.$$Set $d_j:=b_j$ for each $j\in [m+n]\setminus [m]$.  Let $\mathfrak{d}=(d_1, \ldots, d_{m+n})$. Since$$b_1+\cdots +b_m\leq (N-1)(c_{m+1}+\cdots +c_{m+n})-\ell-1,$$one has $\mathfrak{d}\in (N-1)\Pc$ and $\mathfrak{d'}:=\mathfrak{a}-\mathfrak{d}\in \Pc\setminus\partial\Pc$. Since $\mathfrak{a}=\mathfrak{d}+\mathfrak{d'}$, we are done.

\smallskip

{\bf (Case 2.)} Suppose that$$\sum_{i\in D}(a_i-b_i-1)+\sum_{i\in D'}[(N-1)c_i-b_i]\leq\ell.$$Then
\begin{align*}
& \,-\sum_{i=1}^{m}b_i'+\sum_{i\in D}(a_i-b_i-1)+\sum_{i\in D'}[(N-1)c_i-b_i]\\  \leq&\,\ell-\sum_{i=1}^{m}b_i'
\\
=&\,-(c_{m+1}+\cdots +c_{m+n}).    \end{align*}
Since $b_i+b_i'=a_i$, we conclude that$$-|D|+\sum_{i\in D'}[(N-1)c_i-a_i]\leq -(c_{m+1}+\cdots +c_{m+n}).$$In other words,
\begin{eqnarray}
\label{6}
c_{m+1}+\cdots +c_{m+n}\leq |D|+\sum_{i\in D'}[a_i-(N-1)c_i].
\end{eqnarray}
Since $D\sqcup D'=[m]$, it follows that
\begin{align*}
c_{m+1}+\cdots +c_{m+n} \leq & \, |D|+\sum_{i\in D'}(c_i-1)\\ =& \, |D|+\sum_{i=1}^{m}c_i-\sum_{i\in D}c_i-m+|D|\\  =& \, 2|D|+\sum_{i=1}^mc_i-\sum_{i\in D}c_i-m.
\end{align*}
Furthermore, since $$a_1+\cdots +a_m\leq N(c_{m+1}+\cdots +c_{m+n})-1,$$ by using (\ref{6}), one has
\begin{align*}
c_{m+1}+\cdots +c_{m+n} \leq &\,|D|+\sum_{i=1}^ma_i-\sum_{i\in D}a_i-(N-1)\sum_{i\in D'}c_i\\  \leq &\,|D|+N(c_{m+1}+\cdots +c_{m+n})-1-|D|-(N-1)\sum_{i\in D'}c_i\\  = &\,N(c_{m+1}+\cdots +c_{m+n})-1-(N-1)\sum_{i\in D'}c_i.   
\end{align*} 
Therefore,$$(N-1)\sum_{i\in D'}c_i\leq (N-1)(c_{m+1}+\cdots +c_{m+n})-1,$$and $$\sum_{i\in D'}c_i<c_{m+1}+\cdots +c_{m+n}.$$  The above inequality, in particular shows that $D'$ is a proper subset of $[m]$. Thus, $D\neq\emptyset$. Consequently, $D$ is a subset which denies  condition (ii). \, \, \, \, \, \, \, \, \, 
\end{proof}

\section{Convex polytopes of Veronese type}
In the present section, as a corollary of Theorem \ref{compbip-level}, we obtain a numerical characterization for a convex polytope of Veronese type to be level* (Theorem \ref{veronese_criterion}).

\begin{Definition}
\label{veronese_type}
{\em
Let $(a,\mathfrak{c}) = (a, c_1, \ldots, c_{n})\in (\ZZ_{>0})^{n+1}$ be a sequence for which $a > c_1 \geq \cdots \geq c_{n} \geq 2, a \geq n+1$ and $a < c_1+\cdots + c_{n}$.  The convex polytope $\Qc^\sharp_{(a; \mathfrak{c})} \subset \RR^{n}$ which is defined by the system of linear inequalities 

$\bullet$ $0\leq x_i\leq c_i, \, \, \, \, \, i=1\ldots, n$;

$\bullet$ $x_1+\cdots +x_{n}\leq a$

\noindent
is called a convex polytope {\em of Veronese type}.
}
\end{Definition}

Following the notation in Definition \ref{veronese_type}, we focus on a star graph $G=K_{1,n}$ on $[n+1]$ whose edges are $\{1,2\}, \ldots, \{1,n+1\}$ and $\mathfrak{c}=(c_2, \ldots, c_{n+1})\in (\ZZ_{>0})^{n}$.  It follows from Lemma \ref{compbip-poly1} that $\conv(\Dc(G,(a,\mathfrak{c})) \subset \RR^{n+1}$ is defined by the system of linear inequalities 

$\bullet$ $0\leq x_1\leq a$;

$\bullet$ $0\leq x_i\leq c_i, \, \, \, \, \, i=2\ldots, n+1$;

$\bullet$ $x_2+\cdots +x_{n+1}\leq a$.

\noindent
In particular, $\conv(\Dc(G,\mathfrak{c}))$ is a prism over $\Qc^\sharp_{(a;\mathfrak{c})}$.  Since $\conv(\Dc(G,\mathfrak{c}))$ is normal, it follows that every convex polytope of Veronese type (is a lattice polytope and) is normal.  We easily see that up to a transition, $\Qc^\sharp_{(a; \mathfrak{c})}$ is reflexive if and only if $c_i=2$ for each $1 \leq i \leq n$ and $a=n+1$.  Furthermore, $\Qc^\sharp_{(a, \mathfrak{c})}$ is pseido-Gorenstein* if and only if either $c_i =2$ for each $1 \leq i \leq n$ or $a=n+1$.  On the other hand, since $\conv(\Dc(G,\mathfrak{c}))$ is a prism over $\Qc^\sharp_{(a;\mathfrak{c})}$, it follows that $\Qc^\sharp_{(a, \mathfrak{c})}$ is level* if and only if $\conv(\Dc(G,\mathfrak{c}))$ is level*.  Thus as a corollary of Theorem \ref{compbip-level}, one can obtain a numerical characterization for a convex polytope of Veronese type to be level*. 

\begin{Theorem}
 \label{veronese_criterion} 
A lattice polytope of Veronese type $\Qc^\sharp_{(a; \mathfrak{c})} \subset \RR^{n}$ of Definition \ref{veronese_type} is level* if and only if every subset $\emptyset \neq X\subset [n]$ satisfies neither
\begin{eqnarray*}
\label{aaaaa}
a<\sum_{i\in X}c_i< a-n+2|X|-1    
\end{eqnarray*}
nor
\begin{eqnarray*}
\label{bbbbb}
    \sum_{i=1}^{n}c_i-\sum_{i\in X}c_i < a\leq 2|X|+\sum_{i=1}^{n}c_i-\sum_{i\in X}c_i-n.
\end{eqnarray*}  
\end{Theorem}

From Theorem \ref{veronese_criterion}, one can understand how difficult it is to determine whether a given normal polytope is level* or not.  

\begin{Example}
\label{veronese_COR}
{\em 
Let $(a,\mathfrak{c}) = (a, c, \ldots, c)\in (\ZZ_{>0})^{n+1}$, where $a > c \geq 2, a \geq n+1$ and $a < nc$.  Then $\Qc^\sharp_{(a, \mathfrak{c})} \subset \RR^{n}$ is level* if and only if$$a\notin \bigcup_{k=1}^n[kc+n-2k+2, kc-1]\cup [(n-k)c+1, 2k+(n-k)c-n].$$
In particular, if  
\[
a = \left\{
\begin{array}{ll}
(n/2)c + 1& \text{$(\,n$ is even\,$)$}; \\ \\
((n+1)/2)c & \text{$(\,n$ is odd\,$)$}.
\end{array}
\right.
\]
then $\Qc^\sharp_{(a, \mathfrak{c})}$ is level*.  If $c=2$, then $\Qc^\sharp_{(a, \mathfrak{c})}$ is level* if and only if $a = n+1$.
}
\end{Example}

\begin{Example}
{\em
If $\mathfrak{c} = (3,3,2,2,2)$, then there is {\em no} $6 \leq a < 12$ for which $\Qc^\sharp_{(a, \mathfrak{c})}$ is level*.  
}
\end{Example}

Let $\Pc \subset \RR^n$ be a normal polytope of dimension $n$ with $(\Pc \setminus \partial \Pc) \cap \ZZ^n \neq \emptyset$.  Given $\ab \in N(\Pc \setminus \partial \Pc) \cap \ZZ^n, N=1,2,\ldots\,$, the {\em reduced} $\Pc$-{\em degree} of $\ab$ is defined to be the smallest integer $r \geq 1$ for which there is $\ab_0 \in r(\Pc \setminus \partial \Pc) \cap \ZZ^n$ together with $\ab' \in (N-r)\Pc \cap \ZZ^n$ for which $\ab = \ab_0 + \ab'$.  It is shown \cite[Theorem 4.2]{BH} that the {\em reduced} $\Pc$-{\em degree} of $\ab \in N(\Pc \setminus \partial \Pc) \cap \ZZ^n$ is at most $n-1$.  The {\em int* degree} \cite[p.~222]{BH} of $\Pc$ is defined to be the biggest reduced $\Pc$-degree of $\ab \in N(\Pc \setminus \partial \Pc) \cap \ZZ^n$ with $N \geq 1$.  In particular, $\Pc$ is level* if and only if the int* degree of $\Pc$ is equal to $1$.

\begin{Example}
{\em Let $\mathfrak{c} = (5,3,3,3) \in (\ZZ_{>0})^{4}$.  Since $\Qc^\sharp_{(6; \mathfrak{c})}$ is not level (Theorem \ref{veronese_criterion}), the int* degree of $\Qc^\sharp_{(6; \mathfrak{c})}$ is at least $2$.  Furthermore, its int* degree is at most $3$ (\cite[Theorem 4.2]{BH}).  One easily see that the reduced $\Qc^\sharp_{(6; \mathfrak{c})}$-degree of $(8,1,1,1)$ is $2$ and that of $(14,1,1,1)$ is $3$.  Hence the int* degree of $\Qc^\sharp_{(6; \mathfrak{c})}$ is $3$.
    }
\end{Example}

\begin{Question}
{\em Given integers $2 \leq d \leq n$, is there a convex polytope of Veronese type of dimension $n$ whose int* degree is $d-1$?
    }
\end{Question}

\begin{Conjecture}
\label{int*_degree_conjecture}
{\em If the int* degree of $\Pc$ is $d$, then for each $1 \leq i < d$, there is $\ab \in N(\Pc \setminus \partial \Pc) \cap \ZZ^n$ whose reduced $\Pc$-degree is $i$.} 
\end{Conjecture}

Conjecture \ref{int*_degree_conjecture} arose during the preparation of \cite{BH}.

\begin{Example}
\label{veronese_EX}
{\em
Let $\mathfrak{c} = (n, 2, \ldots, 2) \in (\ZZ_{>0})^{n}$.  It is shown \cite[Example 5.2]{BH} that the int* degree of $\Qc^\sharp_{(n+1; \mathfrak{c})}$ is $n-1$.  In other words, $\Qc^\sharp_{(n+1; \mathfrak{c})}$ is far from being level*.  Furthermore, Conjecture \ref{int*_degree_conjecture} is true for $\Qc^\sharp_{(n+1; \mathfrak{c})}$.  
}
\end{Example}

\section{Labeling pseudo-Gorenstein* trees}
Recall that a finite graph $G$ on $n$ vertices is labeling pseudo-Gorenstein* if there is $\mathfrak{c} \in (\ZZ_{>0})^{n}$ for which $\conv(\Dc(G,\mathfrak{c})) \subset \RR^n$ is pseudo-Gorenstein*.  In the present section, we classify labeling pseudo-Gorenstein* trees.  

\begin{Lemma} \label{tree1}
If a tree $T$ has two leaves $x_1, x_2$ with ${\rm dist}(x_1,x_2)=2$, then $T$ is not labeling pseudo-Gorenstein*.
\end{Lemma}

\begin{proof}
Let $V(T)=\{x_1, \ldots, x_n\}$ denote the vertex set of $T$ and suppose that  $T$ is labeling pseudo-Gorenstein*.  One has $\mathfrak{c} = (c_1, c_2, \ldots, c_n) \in (\ZZ_{>0})^n$ for which $\conv(\Dc(T,\mathfrak{c}))$ is pseudo-Gorenstein*.  Set $\Pc :=\conv(\Dc(T,\mathfrak{c}))$. First, since $(\Pc \setminus \partial \Pc) \cap \ZZ^n \neq \emptyset$, it follows that $c_i\geq 2$, for any integer $i\in n$. For each $x_i\in V(T)$,  write $d_i$ for the number of leaves of $T$ which are adjacent to $x_i$.  In addition, set $U:=\{i\in [n] \mid d_i\geq 2\}$.  Thus, $U\neq\emptyset$.  We consider the following two cases. 

\medskip

{\bf Case 1.} Suppose that there is $i\in U$ with $c_i\leq d_i$.  Let $x_1, x_2, \ldots, x_{d_i}$ denote the leaves of $T$ which are adjacent to $x_i$. Then $\rho_{(T,\mathfrak{c})}([d_i])=c_i$. We show that $[d_i]$ is a $\rho_{(T,\mathfrak{c})}$-closed and $\rho_{(T,\mathfrak{c})}$-inseparable subset of $[n]$.  

To see why $[d_i]$ is $\rho_{(T,\mathfrak{c})}$-closed, we must show the inequality$$\rho_{(T,\mathfrak{c})}([d_i]\cup\{\ell\})> \rho_{(T,\mathfrak{c})}([d_i]),$$ 
for each $\ell\in [n]\setminus [d_i]$.  Let $u\in \Bc(T,\mathfrak{c})$ satisfy  $\sum_{j=1}^{d_i}{\rm deg}_{x_j}(u)=\rho_{(T,\mathfrak{c})}([d_i])$. If $u$ is divisible by $x_{\ell}$, then the inequality$$\rho_{(T,\mathfrak{c})}([d_i]\cup\{\ell\})> \rho_{(T,\mathfrak{c})}([d_i])$$trivially holds.  So, assume that $u$ is not divisible by $x_{\ell}$.  Set $\delta:=\delta_{\mathfrak{c}}(I(T))$. As $u\in (I(T)^{\delta})_{\mathfrak{c}}$, $u$ can be written as $u=e_1\cdots e_{\delta}$, where $e_1, \ldots, e_{\delta}$ are edges of $T$. There is $x_q\in N_T(x_{\ell})$ with $q\neq i$. If $x_q$ does not divide $u$, then $(x_{\ell}x_q)u\in (I(T)^{\delta+1})_{\mathfrak{c}}$, a contradiction. Thus, $x_q$ divides $u$. Let $e_1=\{x_q,x_{q'}\}$. Since $q\neq i$, one has $q'\notin [d_i]$.  Since $
ux_{\ell}/x_{q'}
=(x_{\ell}x_q)e_2\cdots e_{\delta}\in \Bc(T,\mathfrak{c})$, it follows that
\begin{eqnarray*}
\rho_{(T,\mathfrak{c})}([d_i]\cup\{\ell\}) 
& \geq & \sum_{j=1}^{d_i}{\rm deg}_{x_j}(ux_{\ell}/x_{q'})+{\rm deg}_{x_{\ell}}(ux_{\ell}/x_{q'})  \\
& > & \sum_{j=1}^{d_i}{\rm deg}_{x_j}(ux_{\ell}/x_{q'})= \sum_{j=1}^{d_i}{\rm deg}_{x_j}(u)
=\rho_{(T,\mathfrak{c})}([d_i]).
\end{eqnarray*}
Consequently, $[d_i]$ is $\rho_{(T,\mathfrak{c})}$-closed, as desired.

To see why $[d_i]$ is $\rho_{(T,\mathfrak{c})}$-inseparable, suppose that $A_1$ and $A_2$ are nonempty subsets of $[d_i]$ with $A_1 \cap A_2 =\emptyset$ and $A_1\cup A_2=[d_i]$.  If $\sum_{j\in A_1}c_j\geq c_i$, then $\rho_{(T,\mathfrak{c})}(A_1)=c_i$.  Therefore,
\[
\rho_{(T,\mathfrak{c})}(A_1)+\rho_{(T,\mathfrak{c})}(A_2) > c_i=\rho_{(T,\mathfrak{c})}([d_i]).
\]
The argument in the case $\sum_{j\in A_2}c_j\geq c_i$ is similar.  Supposing that $\sum_{i\in A_1}c_i< c_1$ and $\sum_{i\in A_2}c_i < c_1$, one has $\rho_{(T,\mathfrak{c})}(A_k)=\sum_{j\in A_k}c_j$ for $k=1,2$.  Thus, 
$$\rho_{(T,\mathfrak{c})}(A_1)+\rho_{(T,\mathfrak{c})}(A_2)=\sum_{i=1}^{d_i}c_i>d_i\geq c_i=\rho_{(T,\mathfrak{c})}([d_i]).$$   
Hence, $[d_i]$ is $\rho_{(T,\mathfrak{c})}$-inseparable, as desired. 

We conclude from Lemma \ref{facets} that the inequality $x_1+\cdots +x_{d_i}\leq c_i$ is one of the defining inequalities of $\Pc$. Since $\Pc$ is pseudo-Gorenstein*, it follows from this inequality that $c_i>d_i$, contradicting our assumption in this case.

\medskip

{\bf Case 2.} Suppose $c_i>d_i$ for each $i\in U$. 

\smallskip

\noindent
{\bf (Subcase 2.1.)} Suppose that there is $i\in U$ with $d_i<c_i\leq \sum c_j$, where the sum is taken over all $j\in [n]$ for which $x_j$ is a leaf adjacent to $x_i$.  We denote these vertices by $x_1, \ldots, x_{d_i}$. Hence, $d_i<c_i\leq \sum_{j=1}^{d_j}c_j$. Then every monomial belonging to $\Bc(T,\mathfrak{c})$ is of the form $x_i^{c_i}w$, where $w$ is a monomial which is not divisible by $x_i$. Consequently, the singleton $\{i\}$ is the unique $\rho_{(T,\mathfrak{c})}$-closed and $\rho_{(T,\mathfrak{c})}$-inseparable subset of $[n]$ which contains $i$. Furthermore, $\rho_{(T,\mathfrak{c})}(\{i\})=c_i$. Thus, the inequality $0\leq x_i\leq c_i$ is the only defining inequality of $\Pc$ which involves $x_i$. Since $(\Pc \setminus \partial \Pc) \cap \ZZ^n \neq \emptyset$, one has $(1,1, \ldots, 1)\in \Pc \setminus \partial \Pc$.  Since $c_i>d_i\geq 2$, $(2, 1,1, \ldots, 1)\in \Pc \setminus \partial \Pc$, a contradiction.

\smallskip

\noindent
{\bf (Subcase 2.2.)} Suppose that for each $i\in U$, one has $c_i> \sum c_j$, where the sum is taken over all $j\in [n]$ for which $x_j$ is a leaf adjacent to $x_i$. Choose $i\in U$ (which exists as $U\neq\emptyset$). Assume that $x_1, \ldots, x_{d_i}$ are the leaves of $T$ which are adjacent to $x_i$. Let $A$ be an arbitrary $\rho_{(T,\mathfrak{c})}$-closed and $\rho_{(T,\mathfrak{c})}$-inseparable subset of $[n]$ with $i\in A$. There is $v\in \Bc(T,\mathfrak{c})$ with $\sum_{j\in A}{\rm deg}_{x_j}(v)=\rho_{(T,\mathfrak{c})}(A)$. For any such $v$, set $A_v:=\{j\in A : x_j \ {\rm does \ not \ divide} \ v\}$.

\medskip

{\bf Claim.} If $v\in \Bc(T,\mathfrak{c})$ with $\sum_{j\in A}{\rm deg}_{x_j}(v)=\rho_{(T,\mathfrak{c})}(A)$ sasisties $A_v\neq \emptyset$, then there is $v'\in \Bc(T,\mathfrak{c})$ with $\sum_{j\in A}{\rm deg}_{x_j}(v')=\rho_{(T,\mathfrak{c})}(A)$ and with $|A_{v'}|\leq |A_v|-1$.

\medskip

{\it Proof of the claim.} Since $v\in \Bc(T,\mathfrak{c})$, one has $v=f_1\cdots f_{\delta}$, where $f_1, \ldots, f_{\delta}$ are edges of $T$. Choose $\ell\in A_v$ (which exists as $A_v\neq\emptyset$) and let $x_{\ell'}\in N_T(x_{\ell})$. If ${\rm deg}_{x_{\ell'}}(v) <c_{\ell'}$, then $(x_{\ell}x_{\ell'})v\in (I(T)^{\delta+1})_{\mathfrak{c}}$, a contradiction. Thus, ${\rm deg}_{x_{\ell'}}(v) =c_{\ell'}\geq 2$. Set $\ell_0:=\ell'$. Assume that $f_1, \ldots, f_{c_{\ell'}}$ are incident to $x_{\ell'}$ and that $f_{c_{\ell'}+1}, \ldots, f_{\delta}$ are not incident to $x_{\ell'}$. So, there are (not necessarily distinct) vertices $x_{p_1}, \ldots, x_{p_{c_{\ell'}}}\in V(T)$ with $f_k=\{x_{\ell'}, x_{p_k}\}$ for each $1 \leq k \leq c_{\ell'}$. Since $x_{\ell}$ does not divide $v$, one has $x_{p_k}\neq x_{\ell}$ for each $1 \leq k \leq c_{\ell'}$. If there is $1 \leq k \leq c_{\ell'}$ with $p_k\notin A$, then set$$v':=\frac{x_{\ell}v}{x_{p_k}}=(x_{\ell}x_{\ell'})f_1\cdots f_{k-1}f_{k+1}\cdots  f_{\delta}\in \Bc(T, \mathfrak{c}).$$Since $\ell \in A$ and $p_k\notin A$, we deduce that$$\sum_{j\in A}{\rm deg}_{x_j}(v')> \sum_{j\in A}{\rm deg}_{x_j}(v)=\rho_{(T,\mathfrak{c})}(A),$$a contradiction. Thus $p_k\in A$ for each $1 \leq k \leq c_{\ell'}$. If there is $1 \leq k \leq c_{\ell'}$ with ${\rm deg}_{x_{p_k}}(v)\geq 2$, then  $v'$ as defined above satisfies the desired properties. Assume that ${\rm deg}_{x_{p_k}}(v)=1$, for each $1 \leq k \leq c_{\ell'}$ (so, $p_1, \ldots p_{c_{\ell'}}$ are distinct). If $x_{p_1}, \ldots, x_{p_{c_{\ell'}}}$ are leaves of $T$, then $\ell'\in U$ and $c_{\ell'}\leq \sum c_j$, where the sum is
taken over all $j\in [n]$ for which $x_j$ is a leaf of $T$ adjacent to $x_{\ell'}$. This contradicts our assumption in this subcase. So, assume that $x_{p_1}$ is not a leaf of $T$. Set $\ell_1:=p_1$. In particular, ${\rm deg}_{x_{\ell_1}}(v)=1$ and $f_1=\{x_{\ell_0}, x_{\ell_1}\}$. Since $x_{\ell_1}$ is not a leaf of $T$, there is $x_{\ell_2}\in N_T(x_{\ell_1})$ with $x_{\ell_2}\neq x_{\ell_0}$. If ${\rm deg}_{x_{\ell_2}}(v)<c_{\ell_2}$, then $(x_{\ell_1}x_{\ell_2})v\in (I(T)^{\delta+1})_{\mathfrak{c}}$, a contradiction. Thus, ${\rm deg}_{x_{\ell_2}}(v) =c_{\ell_2}$. Let $f_{r_1}, \ldots, f_{r_{c_{\ell_2}}}$ be incident to $x_{\ell_2}$. Therefore, there are (not necessarily distinct) vertices $x_{q_1}, \ldots, x_{q_{c_{\ell_2}}}\in V(T)$ for which $f_{r_k}=\{x_{\ell_2}, x_{q_k}\}$, for each $1 \leq k \leq c_{\ell_2}$. Since $x_{\ell}$ does not divide $v$, one has $x_{q_k}\neq x_{\ell}$, for each $1 \leq k \leq c_{\ell_2}$. If there is $1 \leq k \leq c_{\ell_2}$ with $q_k\notin A$, then set$$v'':=\frac{x_{\ell}v}{x_{q_k}}=(x_{\ell}x_{\ell_0})(x_{\ell_1}x_{\ell_2})\frac{f_1\cdots f_{\delta}}{f_1f_{r_k}}\in \Bc(T, \mathfrak{c}).$$Recall that $f_1=\{x_{\ell_0}, x_{\ell_1}\}$. Since $\ell \in A$ and $q_k\notin A$, we conclude that$$\sum_{j\in A}{\rm deg}_{x_j}(v'')> \sum_{j\in A}{\rm deg}_{x_j}(v)=\rho_{(T,\mathfrak{c})}(A),$$a contradiction.  Let $q_k\in A$ for each $1 \leq k \leq c_{\ell_2}$. If there is $1 \leq k \leq c_{\ell_2}$ with ${\rm deg}_{x_{q_k}}(v)\geq 2$, then  $v''$ as defined above satisfies the desired properties. Assume that ${\rm deg}_{x_{q_k}}(v)=1$ for each $1 \leq k \leq c_{\ell_2}$ (so, $q_1, \ldots q_{c_{\ell_2}}$ are distinct). If $x_{q_1}, \ldots, x_{q_{c_{\ell_2}}}$ are leaves of $T$, then $\ell_2\in U$ and $c_{\ell_2}\leq \sum c_j$, where the sum is
taken over all $j\in [n]$ for which $x_j$ is a leaf of $T$ adjacent to $x_{\ell_2}$. This contradicts our assumption in this subcase. We may assume that $x_{q_1}$ is not a leaf of $T$. Set $\ell_3:=q_1$. Continuing this process, we obtain a path $P : x_{\ell_0}, x_{\ell_1}, x_{\ell_2}, x_{\ell_3} \ldots, $ in $T$ for which

($\ast$) ${\rm deg}_{x_{\ell_j}}(v)=c_{\ell_j}$ if $j$ is even, and ${\rm deg}_{x_{\ell_j}}(v)=1$, otherwise.

\noindent
Notice that since $T$ has no cycle, the vertices of $P$ are distinct. As $|V(T)|<\infty$, this process must stop, which completes the proof of the claim. 

\medskip

Now, the claim guarantees that there is $v\in \Bc(T,\mathfrak{c})$ with $\sum_{j\in A}{\rm deg}_{x_j}(v)=\rho_{(T,\mathfrak{c})}(A)$ and $A_v=\emptyset$. Recall that $x_1, \ldots, x_{d_i}$ are the leaves of $T$
which are adjacent to $x_i$. Since ${\rm deg}_{x_i}(v)\geq c_1+\cdots +c_{d_i}\geq 4$ and since $A_v=\emptyset$, one has $$\rho_{(T,\mathfrak{c})}(A)=\sum_{j\in A}{\rm deg}_{x_j}(v)\geq |A|+3.$$
Thus, any defining inequality of $\Pc$ involving $x_i$ is of the form $\sum_{j\in A}x_j\leq t$, where $A$ is a subset of $[n]$ with $i\in A$ and $t\geq |A|+3$ is an integer.  Since $(\Pc \setminus \partial \Pc) \cap \ZZ^n \neq \emptyset$, one has $(1,1, \ldots, 1)\in \Pc \setminus \partial \Pc$.  On the other hand, it follows from our discussion that $(2,1, \ldots, 1)\in \Pc \setminus \partial \Pc$, a contradiction. 
\, \, \, \, \, \, \, \, \, \, \, \, \, \, \, \, \, \, \, \, \, \, \, \,
\end{proof}

\begin{Lemma} \label{tree2}
Let $T$ be a tree on $n\geq 2$ vertices with the property that, for any two leaves $x_1, x_2$ of $T$, one has ${\rm dist}(x_1,x_2)\neq 2$.  Then $T$ is labeling pseudo-Gorenstein*.
\end{Lemma}  

\begin{proof}
Let $\mathfrak{c}:=(2,2 \ldots, 2) \in (\ZZ_{>0})^n$ and set $\Pc :=\conv(\Dc(T,\mathfrak{c}))$. We show that $\Pc$ is pseudo-Gorenstein*. 

\medskip

{\bf Claim 1.} If $\emptyset \neq A\subset [n]$, then there is $v\in \Bc(T,\mathfrak{c})$ with $\sum_{j\in A}{\rm deg}_{x_j}(v)=\rho_{(T,\mathfrak{c})}(A)$ for which $\prod_{j\in A}x_j$ divides $v$.

\medskip

{\it Proof of Claim 1.} Choose $v\in \Bc(T,\mathfrak{c})$ with $\sum_{j\in A}{\rm deg}_{x_j}(v)=\rho_{(T,\mathfrak{c})}(A)$ for which $$A_v:=\{j\in A : x_j \ {\rm does \ not \ divide} \ v\}$$has the smallest cardinality. We prove that $A_v=\emptyset$.

Let $A_v\neq \emptyset$ and set $\delta:=\delta_{\mathfrak{c}}(I(T))$. Since $v\in \Bc(T,\mathfrak{c})$, one has $v=f_1\cdots f_{\delta}$, where $f_1, \ldots, f_{\delta}$ are edges of $T$. Choose $\ell\in A_v$ (which exists as $A_v\neq\emptyset$). Let $x_{\ell'}\in N_T(x_{\ell})$. If ${\rm deg}_{x_{\ell'}}(v) <2$, then $(x_{\ell}x_{\ell'})v\in (I(T)^{\delta+1})_{\mathfrak{c}}$, a contradiction. Thus, ${\rm deg}_{x_{\ell'}}(v)=2$. Set $\ell_0:=\ell'$. Suppose that $f_{i_1}, f_{i_2}$ are incident to $x_{\ell_0}$. Therefore, there are (not necessarily distinct) vertices $x_{p_1}, x_{p_2}\in V(T)$ for which $f_{i_k}=\{x_{\ell_0}, x_{p_k}\}$ for $k=1,2$. Since $x_{\ell}$ does not divide $v$, one has $x_{p_k}\neq x_{\ell}$ for $k=1,2$. If $p_1\notin A$, then set$$v':=\frac{x_{\ell}v}{x_{p_1}}=(x_{\ell}x_{\ell_0})\frac{f_1\cdots f_{\delta}}{f_{i_1}}\in \Bc(T, \mathfrak{c}).$$Since $\ell \in A$ and $p_k\notin A$, we deduce that$$\sum_{j\in A}{\rm deg}_{x_j}(v')> \sum_{j\in A}{\rm deg}_{x_j}(v)=\rho_{(T,\mathfrak{c})}(A),$$a contradiction. Therefore, $p_1\in A$. Similarly, $p_2\in A$. If ${\rm deg}_{x_{p_1}}(v)=2$, then for $v'$ defined as above, one has $|A_{v'}|< |A_v|$, a contradiction. Hence, ${\rm deg}_{x_{p_1}}(v)=1$. Similarly, ${\rm deg}_{x_{p_2}}(v)=1$ (so, $p_1, p_2$ are distinct). By assumption, at least one of $x_{p_1}$ and $x_{p_2}$, say $x_{p_1}$, is not a leaf of $T$. Set $\ell_1:=p_1$. In particular, ${\rm deg}_{x_{\ell_1}}(v)=1$. By relabeling $f_1, \ldots, f_{\delta}$ (if necessary), we assume that $f_1=\{x_{\ell_0}, x_{\ell_1}\}$. Since $x_{\ell_1}$ is not a leaf of $T$, there is $x_{\ell_2}\in N_T(x_{\ell_1})$ with $x_{\ell_2}\neq x_{\ell_0}$. If ${\rm deg}_{x_{\ell_2}}(v)<2$, then $(x_{\ell_1}x_{\ell_2})v\in (I(T)^{\delta+1})_{\mathfrak{c}}$, a contradiction. Thus, ${\rm deg}_{x_{\ell_2}}(v)=2$. Let $f_{r_1}, f_{r_2}$ be incident to $x_{\ell_2}$.
Therefore, there are (not necessarily distinct) vertices $x_{q_1}, x_{q_2}\in V(T)$ for which $f_{r_k}=\{x_{\ell_2}, x_{q_k}\}$ for $k=1,2$. Since $x_{\ell}$ does not divide $v$, one has $x_{q_k}\neq x_{\ell}$ for $k=1,2$. If $q_1\notin A$, then set$$v'':=\frac{x_{\ell}v}{x_{q_1}}=(x_{\ell}x_{\ell_0})(x_{\ell_1}x_{\ell_2})\frac{f_1\cdots f_{\delta}}{f_1f_{r_k}}\in \Bc(T, \mathfrak{c}).$$Since $\ell \in A$ and $q_1\notin A$, we conclude that$$\sum_{j\in A}{\rm deg}_{x_j}(v'')> \sum_{j\in A}{\rm deg}_{x_j}(v)=\rho_{(T,\mathfrak{c})}(A),$$a contradiction. Hence, $q_1\in A$. Similarly, $q_2\in A$. If ${\rm deg}_{x_{q_1}}(v)=2$, then for $v''$ as defined above, one has $|A_{v''}|<|A_v|$, a contradiction. So, ${\rm deg}_{x_{q_1}}(v)=1$. Similarly, ${\rm deg}_{x_{q_2}}(v)=1$ (so, $q_1, q_2$ are distinct). By assumption, at least one of $x_{q_1}$ and $x_{q_2}$, say $x_{q_1}$, is not a leaf of $T$. Set $\ell_3:=q_1$. Continuing this process yields a path $P : x_{\ell_0}, x_{\ell_1}, x_{\ell_2}, x_{\ell_3} \ldots,$ in $T$ for which 

($\ast$) ${\rm deg}_{x_{\ell_j}}(v)=2$ if $j$ is even, and ${\rm deg}_{x_{\ell_j}}(v)=1$, otherwise.

\noindent
Notice that since $T$ has no cycle, the vertices of $P$ are distinct. As $|V(T)|<\infty$, this process must stop, which completes the proof of Claim 1.

\medskip

{\bf Claim 2.} For any nonempty subset $A\subset [n]$, one has $\rho_{(T,\mathfrak{c})}(A)\geq |A|+1$.

\medskip

{\it Proof of Claim 2.} (The proof is similar to that of Claim 1.)  Claim 1 guarantees that there is $v\in \Bc(T,\mathfrak{c})$ with $\sum_{j\in A}{\rm deg}_{x_j}(v)=\rho_{(T,\mathfrak{c})}(A)$ for which $\prod_{j\in A}x_j$ divides $v$. So, one has $\rho_{(T,\mathfrak{c})}(A)\geq |A|$. To prove Claim 2, we only need to show that there is $j\in A$ for which $x_j^2$ divides $v$. As above, let $v=f_1\cdots f_{\delta}$, where $f_1, \ldots, f_{\delta}$ are edges of $T$. Choose $\ell\in A$ and suppose ${\rm deg}_{x_{\ell}}(v)=1$. Next, choose $x_{\ell'}\in N_T(x_{\ell})$ as follows. If $x_{\ell}$ is a leaf of $T$, then $x_{\ell'}$ is its unique neighbor. Otherwise, as ${\rm deg}_{x_{\ell}}(v)=1$, there is $x_{\ell'}\in N_T(x_{\ell})$ for which the edge $\{x_{\ell}, x_{\ell'}\}$ is  not equal to any of $f_1, \ldots, f_{\delta}$. If ${\rm deg}_{x_{\ell'}}(v) <2$, then $(x_{\ell}x_{\ell'})v\in (I(T)^{\delta+1})_{\mathfrak{c}}$, a contradiction. Thus, ${\rm deg}_{x_{\ell'}}(v)=2$. Set $\ell_0:=\ell'$. Suppose that $f_{i_1}, f_{i_2}$ are incident to $x_{\ell_0}$. There are (not necessarily distinct) vertices $x_{p_1}, x_{p_2}\in V(T)$ for which $f_{i_k}=\{x_{\ell_0}, x_{p_k}\}$ for $k=1,2$. If $p_1\notin A$, then set$$v':=\frac{x_{\ell}v}{x_{p_1}}=(x_{\ell}x_{\ell_0})\frac{f_1\cdots f_{\delta}}{f_{i_1}}\in \Bc(T, \mathfrak{c}).$$Since $\ell \in A$ and $p_1\notin A$, we deduce that$$\sum_{j\in A}{\rm deg}_{x_j}(v')> \sum_{j\in A}{\rm deg}_{x_j}(v)=\rho_{(T,\mathfrak{c})}(A),$$a contradiction. Hence, $p_1\in A$. Similarly, $p_2\in A$. If ${\rm deg}_{x_{p_1}}(v)=2$ or ${\rm deg}_{x_{p_2}}(v)=2$, then we are done. Assume that ${\rm deg}_{x_{p_1}}(v)=1$ and ${\rm deg}_{x_{p_2}}(v)=1$ (so, $p_1, p_2$ are distinct). At least one of $x_{p_1}$ and $x_{p_2}$ is not a leaf of $T$, say, $x_{p_1}$.  Furthermore, by the choice of $x_{\ell'}$, one has $x_{p_1}\neq x_{\ell}$. Set $\ell_1:=p_1$. By a similar argument as in the proof of Claim 1, we obtain a path $P : x_{\ell_0}, x_{\ell_1}, x_{\ell_2}, x_{\ell_3} \ldots,$ in $T$ for which 

($\ast$) ${\rm deg}_{x_{\ell_j}}(v)=2$ if $j$ is even, and ${\rm deg}_{x_{\ell_j}}(v)=1$, otherwise.

\noindent
As $|V(T)|<\infty$, this process must stop, which completes the proof of the Claim 2.

\medskip

It follows from Claim 2 and Lemma \ref{facets} that $(1, 1, \ldots,1) \in (\Pc \setminus \partial \Pc) \cap \ZZ^n$.  On the other hand, since $\mathfrak{c}=(2, \ldots, 2)$, one has $\{(1, 1, \ldots,1)\} = (\Pc \setminus \partial \Pc) \cap \ZZ^n$. Hence, $\Pc$ is pseudo-Gorenstein*, as desired.
\hspace{8cm}
\end{proof}

Lemmata \ref{tree1} and \ref{tree2} yield a classification of labeling pseudo-Gorenstein* trees.

\begin{Theorem} \label{treeclassify}
Let $T$ be a tree on $n\geq 2$ vertices.  Then the following conditions are equivalent:
\begin{itemize}
\item [(i)] $T$ is labeling pseudo-Gorenstein*.
\item [(ii)] If $x_1$ and $x_2$ are leaves
of $T$, then ${\rm dist}(x_1, x_2)\neq 2$.
\end{itemize}
\end{Theorem}

\begin{Corollary}
A path $P_n$ on $n \geq 2$ vertices is labeling pseudo-Gorenstein* if and only if $n \neq 3$.
\end{Corollary}

\section*{Appendix}
Let $\Pc \subset \RR^n$ be a lattice polytope of dimension $n$ and $i(\Pc, N) = |N\Pc \cap \ZZ^n|$ for $N = 1,2, \ldots.$.  A fundamental result by Ehrhart says that $i(\Pc, N)$ is a polynomial in $N$ of degree $n$ with $i(\Pc,0) = 1$.  A fact on generating functions guarantees that
\begin{eqnarray*}
\label{delta}
\delta(\Pc, \lambda) = (1 - \lambda)^{n+1} \left[ 1 + \sum_{N=1}^\infty i(\Pc, N) \lambda^N \right]
\end{eqnarray*}   
is a polynomial in $\lambda$ of degree $\leq n$.  Let $\delta(\Pc, \lambda) = \sum_{i=0}^{n} \delta_i \lambda^i$.  We say that $$\delta(\Pc) = (\delta_0, \delta_1, \ldots, \delta_n)$$ is the {\em $\delta$-vector} \cite{Hibi_DM} of $\Pc$. 

Now, it follows from \cite[Corollaries 2.2, 2.3]{APPS} that, when $\Pc \subset \RR^n$ is level*, then $\delta(\Pc)$ is {\em unimodal}.  More precisely, $\delta(\Pc)$ satisfies
\begin{eqnarray}
\label{unimodal}
\delta_0 \leq \delta_1 \leq \cdots \leq \delta_{[n/2]} \geq \delta_{[n/2]+1} \geq \cdots \geq \delta_n.
\end{eqnarray} 
The unimodality of $\delta$-vectors strongly encourages combinatorialists to find a natural class of level* polytopes. 
We refer the reader to \cite{APPS} for the historical background on the study of unimodal $\delta$-vectors which originated in a conjecture proposed in \cite{HibiRedBook}.   

\section*{Acknowledgments}
The second author is supported by a FAPA grant from Universidad de los Andes.

\section*{Statements and Declarations}
The authors have no Conflict of interest to declare that are relevant to the content of this article.

\section*{Data availability}
Data sharing does not apply to this article as no new data were created or analyzed in this study.

\end{document}